\ifdefined\journalstyle

\documentclass{svproc}
\usepackage{url}

\usepackage{amsfonts,amsmath}
\newcommand{\bsx}{{\boldsymbol{x}}}
\newcommand{\bsy}{{\boldsymbol{y}}}
\newcommand{\bsnu}{{\boldsymbol{\nu}}}
\newcommand{\bsb}{{\boldsymbol{b}}}
\newcommand{\bbN}{{\mathbb{N}}}
\newcommand{\parder}[2]{\frac{\partial #1}{\partial #2}}
\newcommand{\norm}[1]{\lVert #1 \rVert}
\usepackage{color}

\else
\documentclass[11pt,a4paper]{article}
\usepackage{pdfsync} 
\usepackage{microtype}
\usepackage{amsfonts,amsmath,amsthm}
\newtheorem{theorem}{Theorem}
\newtheorem{lemma}[theorem]{Lemma}

\theoremstyle{definition}

\usepackage{amssymb}
\usepackage{hyperref}
\usepackage{graphicx}
\usepackage{wrapfig}
\usepackage{subfig}
\usepackage{epsfig}
\usepackage{float,mathrsfs}
\usepackage{enumerate}
\usepackage{enumitem}
\usepackage{bbm}

\usepackage{algpseudocode}
\usepackage[font=scriptsize]{caption}
\counterwithout{figure}{section}
\usepackage{color}
\definecolor{refkey}{rgb}{0.9451,0.2706,0.4941}\definecolor{labelkey}{rgb}{0.9451,0.2706,0.4941}

\definecolor{darkred}{RGB}{139,0,0}
\definecolor{darkgreen}{RGB}{0,100,0}
\definecolor{darkmagenta}{RGB}{139,0,139}
\definecolor{gray}{RGB}{180,180,180}

\newcommand{\bsx}{{\boldsymbol{x}}}
\newcommand{\bsy}{{\boldsymbol{y}}}

\newcommand{\bbN}{{\mathbb{N}}}

\newcommand{\bsnu}{{\boldsymbol{\nu}}}

\newcommand{\bsb}{{\boldsymbol{b}}}

\newcommand{\mask}[1]{{}}

\usepackage{pgfplots}
\pgfplotsset{compat=1.18}
\title{On the optimality of dimension truncation error rates for\\a class of parametric partial differential equations}
\author{
Philipp A.~Guth\footnotemark[2] \and Vesa Kaarnioja\footnotemark[3]
}
\date{\today}

\newcommand{\norm}[1]{\lVert #1 \rVert}

\newcommand{\parder}[2]{\frac{\partial #1}{\partial #2}}
\fi

\begin{document}\ifdefined\journalstyle
\mainmatter              
\title{On the optimality of dimension truncation error rates for a class of parametric partial differential equations}
\titlerunning{On the optimality of dimension truncation error rates}  
%
\author{Philipp A. Guth\inst{1} \and Vesa Kaarnioja\inst{2}}
\authorrunning{Philipp A. Guth and Vesa Kaarnioja} 
%
\tocauthor{Philipp A. Guth and Vesa Kaarnioja}
\institute{Johann Radon Institute for Computational and Applied Mathematics, Austrian Academy of Sciences, Altenbergerstra{\ss}e 69, AT-4040 Linz, Austria,\\
\email{philipp.guth@ricam.oeaw.ac.at},
\and
School of Engineering Sciences, LUT University, P.O.~Box 20, 53851 Lappeenranta, Finland,\\
\email{vesa.kaarnioja@lut.fi}}

\maketitle              

\else 
\renewcommand{\thefootnote}{\fnsymbol{footnote}}
\footnotetext[2]{Johann Radon Institute for Computational and Applied Mathematics, Austrian Academy of Sciences, Altenbergerstra{\ss}e 69, AT-4040 Linz, Austria ({\tt philipp.guth@ricam.oeaw.ac.at}).}
\footnotetext[3]{School of Engineering Sciences, LUT University, P.O.~Box 20, 53851 Lappeenranta, Finland ({\tt vesa.kaarnioja@lut.fi}).}
\maketitle

\fi
\begin{abstract}
In uncertainty quantification for parametric partial differential equations (PDEs), it is common to model uncertain random field inputs using countably infinite sequences of independent and identically distributed random variables. The lognormal random field is a prime example of such a model. While there have been many studies assessing the error in the PDE response that occurs when an infinite-dimensional random field input is replaced with a finite-dimensional random field, there do not seem to be any analyses in the existing literature discussing the sharpness of these bounds. This work seeks to remedy the situation. Specifically, we investigate two model problems where the existing dimension truncation error rates can be shown to be sharp.
\ifdefined\journalstyle
\keywords{partial differential equation, lognormal random field, random coefficient}
\else
\fi
\end{abstract}
\section{Introduction}
A frequently studied problem in uncertainty quantification for partial differential equations (PDEs) subject to model uncertainties concerns the problem of finding $u\!:D\times U\to \mathbb R$ such that
$$
\begin{cases}
-\nabla\cdot (a(\bsx,\bsy)\nabla u(\bsx,\bsy))=f(\bsx),&\bsx\in D,~\bsy\in U,\\
u(\cdot,\bsy)|_{\partial D}=0,&\bsy\in U,
\end{cases}
$$
in a bounded, nonempty Lipschitz domain $D\subset\mathbb R^d$, $d\in\{1,2,3\}$, where $a\colon D\times U\to\mathbb R$ is a parameterized diffusion coefficient, $f\!:D\to\mathbb R$ a fixed source term, and the parameter set $U$ is a nonempty subset of $\mathbb R^{\mathbb N}$. Some natural quantities to investigate are the expected value
\begin{align}\label{eq:moti1}
\mathbb E[u(\bsx,\cdot)]=\int_U u(\bsx,\bsy)\,\mu({\rm d}\bsy)
\end{align}
or
\begin{align}\label{eq:moti2}
\mathbb E[G(u)]=\int_U G(u(\cdot,\bsy))\,\mu({\rm d}\bsy),
\end{align}
where $G\!:H_0^1(D)\to\mathbb R$ is a continuous linear quantity of interest and $\mu$ denotes a probability measure over $U$. By $H_0^1(D)$ we denote the subspace of $H^1(D)$ with zero trace on $\partial D$. The measure $\mu$ is typically chosen either as the uniform probability measure over $U=[-1,1]^{\mathbb N}$ or as a Gaussian probability measure over $U=\mathbb R^{\mathbb N}$.

For the numerical approximation of~\eqref{eq:moti1} or~\eqref{eq:moti2}, the first step usually involves truncating the infinite-dimensional parameterization of the input field $a(\cdot,\bsy)$ into a finite number of variables, that is, the input is replaced by a finite-dimensional parametric coefficient~$a_s(\cdot,\bsy):=a(\cdot,(y_1,\ldots,y_s,0,0,\cdots))$. Denoting the corresponding PDE solution by $u_s(\cdot,\bsy):=u(\cdot,(y_1,\ldots,y_s,0,0,\ldots))$, considering $\mathbb E[u_s]$ instead of $\mathbb E[u]$ introduces a \emph{dimension truncation error}
\begin{align}
\|\mathbb E[u-u_s]\|_{H^1(D)}.\label{eq:dimtruncmoti}
\end{align}

Dimension truncation error rates have so far been primarily analyzed for integration problems. A first upper bound for the dimension truncation rate was established in~\cite{kss12} for a class of elliptic PDEs with an affine parameterization of the diffusion coefficient. This result was subsequently improved in~\cite{gantner} within the broader framework of affine-parametric operator equations. Further studies have addressed dimension truncation for coupled PDE systems arising in optimal control under uncertainty~\cite{guth21}, within the periodic setting of uncertainty quantification for numerical integration~\cite{kks20} and kernel interpolation~\cite{kkkns22}, as well as for Bayesian inverse problems governed by PDEs~\cite{dglgs19,hks21}. The analyses in these works rely on Neumann series expansions (see, e.g., \cite[Example 4.5]{Kato}), a technique that performs well in affine-parametric settings but may yield suboptimal estimates when the parameter dependence is nonlinear. Guth and Kaarnioja~\cite{gk22} subsequently generalized the dimension truncation error analysis to a larger problem class using Taylor series arguments. In particular, the rate obtained for~\eqref{eq:dimtruncmoti} in~\cite{gk22} is~$\mathcal{O}(s^{-{2}/{p}+1})$, where~$s$ is the truncation dimension and~$p$ is related to the decay of the parametric input via$$\bigg\|\prod_{j\ge 1} \frac{\partial^{\nu_j}}{\partial y_j^{\nu_j}} a(\cdot,\bsy) \bigg\|_{L^\infty(D)} \le \Gamma_{|\bsnu|} \prod_{j \ge1} b_j^{\nu_j},\quad |\bsnu|:=\sum_{j\geq 1}\nu_j<\infty,$$with~$\bsb = (b_j)_{j \geq 1} \in \ell^p(\bbN)$ for some $p\in(0,1)$ and $(\Gamma_j)_{j\geq 0}\in\ell^\infty(\mathbb N_0)$. Dimension truncation error rates were later established for~$L^2$ function approximation in~\cite{gkmcqmc}. The aforementioned works provide upper bounds for the dimension truncation error, leaving the question of optimality of these rates unaddressed.

In this manuscript we demonstrate that the dimension truncation error rate is sharp by providing concrete examples of parametric PDEs which achieve the theoretical rate exactly.

\section{Model problems}
We show that the dimension truncation rates obtained in~\cite{gk22} are sharp for two classes of elliptic PDE problems, which we introduce below.

\paragraph{Model problem 1.} Let $D=(0,1)$. We consider a diffusion coefficient $\alpha(\cdot,\bsy)\in L^\infty(D)$ of the form
$$
\alpha(x,\bsy):=\exp\bigg(\sum_{j=1}^\infty y_j\psi_j(x)\bigg),\quad x\in D,~\bsy\in U,
$$
where $\psi_j\in L^\infty(D)$ for all $j\geq 1$ such that $(\|\psi_j\|_{L^\infty(D)})_{j\geq1}\in\ell^1(\mathbb N)$, and we define
\begin{align}\label{eq:U1}
U:=\bigg\{\bsy\in\mathbb R^{\mathbb N} : \sum_{j\geq 1}|y_j|\|\psi_j\|_{L^\infty(D)}<\infty\bigg\}.
\end{align}
We consider the solution $v\!:D\times U\to\mathbb R$ to the parametric Dirichlet--Neumann problem
$$
\begin{cases}
-\displaystyle\frac{\partial}{\partial x}\bigg(\alpha(x,\bsy)\frac{\partial}{\partial x}v(x,\bsy)\bigg)=f(x),\quad x\in D,~\bsy\in U,\\
v(0,\bsy)=0=v_x(1,\bsy),\quad \bsy\in U,
\end{cases}
$$
where $f\in L^2(D)$ is the source term.

In this case, it is possible to write down the solution to the parametric Dirichlet--Neumann problem as
\begin{align}\label{eq:solution}
v(x,\bsy)=\int_0^x\bigg(\int_\xi^1 f(z)\,{\rm d}z\bigg)\exp\bigg(-\sum_{j=1}^\infty y_j\psi_j(\xi)\bigg)\,{\rm d}\xi,\quad x\in D,~\bsy\in U.
\end{align}
This explicit identity will be very useful in the derivation of the lower bound for the dimension truncation error.
\paragraph{Model problem 2.} Let $D\subset \mathbb R^d$ be a nonempty, bounded Lipschitz domain with $d\in\{1,2,3\}$. We consider a diffusion coefficient $\beta(\bsy)$ of the form
$$
\beta(\bsy):=\exp\bigg(\sum_{j=1}^\infty b_jy_j\bigg),\quad \bsy\in U,
$$
where $\boldsymbol b=(b_j)_{j\geq 1}\in\ell^1(\mathbb N)$ is a sequence of nonnegative real numbers. Moreover, we define
\begin{align}\label{eq:U2}
U:=\bigg\{\bsy\in \mathbb R^{\mathbb N}: \sum_{j\geq 1}b_j|y_j|<\infty\bigg\}.
\end{align}
We consider the solution $w\!:D\times U\to\mathbb R$ to the parametric Dirichlet problem
$$
\begin{cases}
-\nabla\cdot(\beta(\bsy)\nabla w(\bsx,\bsy))=f(\bsx),\quad \bsx\in D,~\bsy\in U,\\
w(\cdot,\bsy)|_{\partial D}=0,
\end{cases}
$$
where $f\in L^2(D)$ is the source term.

In this case, we can characterize the solution as
$$
w(\bsx,\bsy)=\frac{\widetilde{w}(\bsx)}{\beta(\bsy)},
$$
where $\widetilde w\!: D\to\mathbb R$ is the solution to the Poisson problem
$$
\begin{cases}
-\Delta \widetilde w(\bsx)=f(\bsx),\quad \bsx\in D,\\
\widetilde w|_{\partial D}=0.
\end{cases}
$$

\paragraph{Probability measure.}
Both problems are equipped with the Gaussian probability measure
$$
\mu_G(\mathrm d \bsy) = \bigotimes_{j \ge 1} \mathcal{N}(0,1).
$$
It can be shown~\cite[Lemma 2.28]{schwab_gittelson_2011} that~$U$ as defined in~\eqref{eq:U1} and~\eqref{eq:U2} has full measure, i.e.,~$\mu_G(U) =1$. Thus, the domain of integration~$\mathbb{R}^\bbN$ is interchangeable with~$U$.

\section{Upper bound}

Although the upper bound $\mathcal O(s^{-2/p+1})$ for the dimension truncation rate follows for the model problems as a direct application of~\cite{gk22}, we rederive these bounds explicitly for both model problems in the following. An added benefit of this approach is that we obtain explicit constant factors for the dimension truncation upper bounds in both cases.

\begin{lemma}[Dimension truncation upper bound for model problem 1]\label{lemma:upper1} Assume that $(\|\psi_j\|_{L^\infty((0,1))})_{j\ge 1}\in\ell^p(\mathbb N)$ for some $p\in(0,1)$ such that $\|\psi_1\|_{L^\infty((0,1))}\allowbreak\geq \|\psi_2\|_{L^\infty((0,1))}\geq \cdots$. Under the assumptions posed for model problem 1, there holds
$$
\bigg\|\int_U (v(\cdot,\bsy)-v_s(\cdot,\bsy))\,\mu_G({\rm d}\bsy)\bigg\|_{H^1((0,1))}\leq C_1s^{-2/p+1},
$$
where $C_1:=\frac1{\sqrt 2}\|f\|_{L^2((0,1))}{\rm e}^{\frac12\sum_{j=1}^\infty \|\psi_j\|_{L^\infty((0,1))}^2}\big(\sum_{j=1}^\infty \|\psi_j\|_{L^\infty((0,1))}^p\big)^{2/p}$.
\end{lemma}
\begin{proof}
As a consequence of~\eqref{eq:solution}, we may write
\begin{align*}
&\bigg\|{\parder{}{x}\int_{U}(v(\cdot,\bsy)-v_s(\cdot,\bsy))\,\mu_G({\rm d}\bsy)\bigg\|}_{L^2((0,1))}^2\\[3pt]
&=\int_0^1\Big(\int_x^1 f(z)\,{\rm d}z\Big)^2
\Big(\int_{U}(\alpha(x,\bsy)^{-1}-\alpha_s(x,\bsy)^{-1})\,\mu_G({\rm d}\bsy)\Big)^2 {\rm d}x\\
&\leq \|f\|_{L^2((0,1))}^2\int_0^1 \bigg(\int_U(\alpha(x,\bsy)^{-1}-\alpha_s(x,\bsy)^{-1})\,\mu_G({\rm d}\bsy)\bigg)^2\,{\rm d}x,
\end{align*}
where we interchange the order of differentiation and integration in the first step. Then, after an application of the Cauchy--Schwarz inequality, the first factor can be bounded uniformly in $x$ by enlarging the domain of integration.
Meanwhile, there holds
\begin{align*}
&\bigg\|\int_U (v(\cdot,\bsy)-v_s(\cdot,\bsy))\,\mu_G({\rm d}\bsy)\bigg\|_{L^2((0,1))}^2\\
&=\int_0^1 \bigg[\int_U \int_0^x\bigg(\int_\xi^1 f(z)\,{\rm d}z\bigg)\big(\alpha(\xi,\bsy)^{-1}-\alpha_s(\xi,\bsy)^{-1}\big)\,{\rm d}\xi\,\mu_G({\rm d}\bsy)\bigg]^2\,{\rm d}x\\
&=\int_0^1 \bigg[\int_0^x \bigg(\int_\xi^1 f(z)\,{\rm d}z\bigg)\int_U \big(\alpha(\xi,\bsy)^{-1}-\alpha_s(\xi,\bsy)^{-1}\big)\,\mu_G({\rm d}\bsy)\,{\rm d}\xi\bigg]^2\,{\rm d}x\\
&\leq \int_0^1 \bigg[\int_0^x \bigg(\int_\xi^1f(z)\,{\rm d}z\bigg)^2\,{\rm d}\xi\bigg]\\
&\qquad \times \bigg[\int_0^x \bigg(\int_U \big(\alpha(\xi,\bsy)^{-1}-\alpha_s(\xi,\bsy)^{-1}\big)\,\mu_G({\rm d}\bsy)\bigg)^2\,{\rm d}\xi\bigg]\,{\rm d}x\\
&\leq \|f\|_{L^2((0,1))}^2\int_0^1\bigg(\int_U (\alpha(\xi,\bsy)^{-1}-\alpha_s(\xi,\bsy)^{-1})\,\mu_G({\rm d}\bsy)\bigg)^2\,{\rm d}\xi,
\end{align*}
where we applied the Cauchy–Schwarz inequality and bounded the expression from above by enlarging the domains of integration. 
This means that the $H^1$-error satisfies
\begin{align*}
&\bigg\|\int_U (v(\cdot,\bsy)-v_s(\cdot,\bsy))\,\mu_G({\rm d}\bsy)\bigg\|_{H^1((0,1))}\\
&\leq \sqrt{2}\|f\|_{L^2((0,1))}\sqrt{\int_0^1\bigg(\int_U (\alpha(x,\bsy)^{-1}-\alpha_s(x,\bsy)^{-1})\,\mu_G({\rm d}\bsy)\bigg)^2\,{\rm d}x}.
\end{align*}
The inner integral can be evaluated analytically as
\begin{align*}
\int_{U} (\alpha(x,\bsy)^{-1}-\alpha_s(x,\bsy)^{-1})\,\mu_G({\rm d}\bsy)
= {\rm e}^{\frac12\sum_{j=1}^s \psi_j(x)^2}\Big( {\rm e}^{\frac12\sum_{j>s}\psi_j(x)^2}-1\Big),
\end{align*}
where we used the identity $\frac{1}{\sqrt{2\pi}}\int_{-\infty}^\infty {\rm e}^{by}{\rm e}^{-\frac12y^2}\,{\rm d}y={\rm e}^{\frac12 b^2}$ for $b\in\mathbb R$. Then, using ${\rm e}^t-1 \le t {\rm e}^t$ for $t\geq 0$ 
together with the $\ell^p$-summability of $(\norm{\psi_j}_{L^\infty((0,1))})_{j\ge 1}$ gives
\begin{align*}
\bigg|\int_{U} (\alpha(x,\bsy)^{-1}-\alpha_s(x,\bsy)^{-1})\,\mu_G({\rm d}\bsy)\bigg|\leq \frac12 {\rm e}^{\frac12 \sum_{j=1}^\infty \|\psi_j\|_{L^\infty((0,1))}^2}\sum_{j>s}\|\psi_j\|_{L^\infty((0,1))}^2.
\end{align*}
As a consequence of Stechkin's lemma (cf., e.g.,~\cite[Lemma 3.3]{KressnerTobler}), there holds
$$
\sum_{j>s}\|\psi_j\|_{L^\infty((0,1))}^2\leq s^{-2/p+1}\bigg(\sum_{j=1}^\infty \|\psi_j\|_{L^\infty((0,1))}^p\bigg)^{2/p},
$$
which yields the overall error rate
\begin{align*}
&\bigg\|{\int_{U} (v(\cdot,\bsy)-v_s(\cdot,\bsy))\,\mu_G({\rm d}\bsy)}\bigg\|_{H^1((0,1))}\\
&\leq \frac1{\sqrt 2}\|f\|_{L^2((0,1))}{\rm e}^{\frac12 \sum_{j=1}^\infty \|\psi_j\|_{L^\infty((0,1))}^2}\bigg(\sum_{j=1}^\infty \|\psi_j\|_{L^\infty((0,1))}^p\bigg)^{2/p}s^{-2/p+1},
\end{align*}
as desired.\ifdefined\journalstyle\qed\fi\end{proof}

\begin{lemma}[Dimension truncation upper bound for model problem 2]\label{lemma:ubmodel2}Assume that $(b_j)_{j\geq 1}\in\ell^p(\mathbb N)$ for some $p\in(0,1)$ such that $b_1\geq b_2\geq \cdots$. Under the assumptions posed for model problem 2, there holds
$$
\bigg\|\int_U (w(\cdot,\bsy)-w_s(\cdot,\bsy))\,\mu_G({\rm d}\bsy)\bigg\|_{H^1(D)}\leq C_2s^{-2/p+1},
$$
where~$C_2 :=  \|\widetilde{w}\|_{H^1(D)} \frac1{2}{\rm e}^{\frac12 \sum_{j=1}^\infty b_j^2}\big(\sum_{j=1}^\infty b_j^p\big)^{2/p}$.
\end{lemma}
\begin{proof} Now
\begin{align*}
&\bigg\|\int_U(w(\cdot,\bsy)-w_s(\cdot,\bsy))\,\mu_G({\rm d}\bsy)\bigg\|_{H^1(D)}\\
&=\bigg|\int_U (\beta(\bsy)^{-1}-\beta_s(\bsy)^{-1})\,\mu_G({\rm d}\bsy)\bigg|\|\widetilde w\|_{H^1(D)}\\
&=\big|{\rm e}^{\frac12 \sum_{j=1}^s b_j^2}\big({\rm e}^{\frac12 \sum_{j> s}b_j^2}-1\big)\big|\|\widetilde w\|_{H^1(D)}.
\end{align*}
In complete analogy with the proof of Lemma~\ref{lemma:upper1}, we obtain that
\begin{align*}
&\bigg\|\int_U(w(\cdot,\bsy)-w_s(\cdot,\bsy))\,\mu_G({\rm d}\bsy)\bigg\|_{H^1(D)} \\&\le \|\widetilde{w}\|_{H^1(D)} \frac1{2}{\rm e}^{\frac12 \sum_{j=1}^\infty b_j^2}\bigg(\sum_{j=1}^\infty b_j^p\bigg)^{2/p}s^{-2/p+1},
\end{align*}
as desired.\ifdefined\journalstyle\qed\fi\end{proof}
\section{Lower bound}

To demonstrate sharpness, we next establish matching dimension truncation lower bounds for both model problems.
\begin{lemma}[Dimension truncation lower bound for model problem 1]\label{lem:lb1}
Assume that for some constants $c>0$ and $\theta>1$, there holds
\[
\psi_j(x) \ge c\,j^{-\theta} \quad \text{for all } x\in(0,1),\; j\ge1.
\]
Under the assumptions placed for model problem 1, there holds
$$
\bigg\|\int_U (v(\cdot,\bsy)-v_s(\cdot,\bsy))\,\mu_G({\rm d}\bsy)\bigg\|_{H^1((0,1))}\geq \frac{C_f}{2\theta-1} s^{-2\theta+1},
$$
where~$C_f := \frac{c^2}{2}\sqrt{\int_{0}^1 \big(\int_x^1 f(z)\,{\rm d}z\big)^2 \,\mathrm dx}$.
\end{lemma}
\begin{proof} 
Recall that
\begin{align*}
    &\int_U (v(x,\bsy) - v_s(x,\bsy)) \,\mu_G(\mathrm d\bsy) \\&=\int_0^x\bigg(\int_\xi^1 f(z)\,{\rm d}z\bigg) \int_U\left(\alpha(\xi,\bsy)^{-1} - \alpha_s(\xi,\bsy)^{-1}\right)\,\mu_G(\mathrm d\bsy)\,{\rm d}\xi\\
    &= \int_0^x\bigg(\int_\xi^1 f(z)\,{\rm d}z\bigg)  {\rm e}^{\frac12\sum_{j=1}^s \psi_j(\xi)^2}\Big( {\rm e}^{\frac12\sum_{j>s}\psi_j(\xi)^2}-1\Big)\,{\rm d}\xi.
\end{align*}
Since there holds~$\|h\|_{H^1(D)} \ge \|\tfrac{\partial}{\partial x} h\|_{L^2(D)}$ for any $h\in H^1(D)$, we have
\begin{align*}
    &\bigg\|\int_U (v(\cdot,\bsy)-v_s(\cdot,\bsy))\,\mu_G({\rm d}\bsy)\bigg\|_{H^1}^2 \\
    &\ge \int_{0}^1 \bigg(\int_x^1 f(z)\,{\rm d}z\bigg)^2  {\rm e}^{\sum_{j=1}^s \psi_j(x)^2}\Big( {\rm e}^{\frac12\sum_{j>s}\psi_j(x)^2}-1\Big)^2 \,\mathrm dx.
\end{align*}
Noting that $({\rm e}^x-1)^2\ge x^2$ for $x\ge0$, we obtain
\begin{align*}
    \bigg\|\int_U (v(\cdot,\bsy)-v_s(\cdot,\bsy))\,\mu_G({\rm d}\bsy)\bigg\|_{H^1((0,1))}^2 &\ge \int_{0}^1 \bigg(\int_x^1 f(z)\,{\rm d}z\bigg)^2 \frac14\Big(\sum_{j>s}\psi_j(x)^2\Big)^2  \,\mathrm dx\\
    &\ge C_f^2 \Big(\sum_{j>s} j^{-2\theta}\Big)^2. 
\end{align*}
Since
\[
\sum_{j>s} j^{-2\theta} \ge \int_{s+1}^\infty \tau^{-2\theta}\,{\rm d}\tau
= \frac{1}{2\theta-1}(s+1)^{-2\theta+1},
\]
we find that
\[
\bigg\|{\int_{U} (v(\cdot,\bsy)-v_s(\cdot,\bsy))\,\mu_G({\rm d}\bsy)}\bigg\|_{H^1((0,1))}
\ge \frac{C_f}{2\theta-1} \,s^{-2\theta+1}.
\]This proves the assertion.\ifdefined\journalstyle\qed\fi
\end{proof}
As $\theta \searrow 1/p$, the lower bound exhibits the same algebraic rate as the upper bound,
up to multiplicative constants, confirming the asymptotic sharpness of the estimate. Consider for example a sequence $(\psi_j)_{j\geq 1}$ satisfying $\|\psi_j\|_{L^\infty((0,1))}=c j^{-\theta}$ with $\theta>1$. By Lemma~\ref{lem:lb1}, there holds
$$
\bigg\|\int_U (v(\cdot,\bsy)-v_s(\cdot,\bsy))\,\mu_G({\rm d}\bsy)\bigg\|_{H^1((0,1))}\geq \frac{C_f}{2\theta-1}s^{-2\theta+1}.
$$%
Meanwhile, $(\|\psi_j\|_{L^\infty((0,1))})_{j\geq 1}\in\ell^p(\mathbb N)$ for all $p\in(\frac1{\theta},1)$. By Lemma~\ref{lemma:upper1} and choosing $p=\frac{2}{2\theta-\varepsilon}$ for arbitrary $\varepsilon>0$, there holds
$$
\bigg\|\int_U (v(\cdot,\bsy)-v_s(\cdot,\bsy))\,\mu_G({\rm d}\bsy)\bigg\|_{H^1((0,1))}\leq C_1s^{-\frac{2}{p}+1}=C_1s^{-2\theta+1+\varepsilon}.
$$

\begin{lemma}[Dimension truncation lower bound for model problem 2]\label{lem:lb2}
Assume that for some constants $c>0$ and $\theta>1$, there holds
\[
b_j \ge c\,j^{-\theta} \quad \text{for all }  j\ge1.
\]
Under the assumptions placed for model problem 2, there holds
$$
\bigg\|\int_U (w(\cdot,\bsy)-w_s(\cdot,\bsy))\,\mu_G({\rm d}\bsy)\bigg\|_{H^1(D)}\geq \frac{C_{\widetilde{w}}}{2\theta-1} s^{-2\theta+1},
$$
where~$C_{\widetilde{w}} := \frac{c^2\|\widetilde{w}\|_{H^1(D)}}{2}$.
\end{lemma}
\begin{proof}
Using~$w(\bsx,\bsy) = \widetilde{w}(\bsx) \beta(\bsy)^{-1}$, we obtain
\begin{align*}
    &\bigg\|\int_U (w(\bsx,\bsy) - w_s(\bsx,\bsy)) \,\mu_G(\mathrm d\bsy)\bigg\|_{H^1(D)} \\&=\|\widetilde{w}\|_{H^1(D)} \bigg|\int_U\left(\beta(\bsy)^{-1} - \beta_s(\bsy)^{-1}\right)\,\mu_G(\mathrm d\bsy)\bigg|\\
    &= \|\widetilde{w}\|_{H^1(D)} {\rm e}^{\frac12\sum_{j=1}^s b_j^2}\Big( {\rm e}^{\frac12\sum_{j>s}b_j^2}-1\Big)\\
    &\ge \frac{\|\widetilde{w}\|_{H^1(D)}}{2}\sum_{j>s}b_j^2 \\
    &\ge \frac{\|\widetilde{w}\|_{H^1(D)} c^2}{2}\sum_{j>s}j^{-2\theta}, 
\end{align*}
which leads by the same arguments as in the proof of Lemma~\ref{lem:lb1} to the desired result.\ifdefined\journalstyle\qed\fi\end{proof}

The asymptotic sharpness follows by the same reasoning as above: considering $b_j=cj^{-\theta}$ with $\theta>1$, Lemmas~\ref{lemma:ubmodel2} and~\ref{lem:lb2} imply that
$$
\bigg\|\int_U (w(\cdot,\bsy)-w_s(\cdot,\bsy))\,\mu_G({\rm d}\bsy)\bigg\|_{H^1(D)}\geq \frac{C_{\widetilde{w}}}{2\theta-1}s^{-2\theta+1}
$$
and
$$
\bigg\|\int_U (w(\cdot,\bsy)-w_s(\cdot,\bsy))\,\mu_G({\rm d}\bsy)\bigg\|_{H^1(D)}\leq C_2s^{-2\theta+1+\varepsilon}
$$
for arbitrary $\varepsilon>0$.

\section{Conclusion}

Dimension truncation error analyses have been conducted in the existing literature for a wide class of high-dimensional integration problems arising in PDE uncertainty quantification problems. However, these analyses have been primarily concerned with upper bounds. The present paper demonstrates that the upper bounds are sharp by constructing explicit examples, where the dimension truncation rate $\mathcal O(s^{-2/p+1})$ can be shown to be optimal.

\section*{Acknowledgment}

Vesa Kaarnioja was supported by the Research Council of Finland (Flagship of Advanced Mathematics for Sensing, Imaging and Modelling grant 359183).

\ifdefined\journalstyle
\bibliographystyle{spmpsci}
\else
\bibliographystyle{plain}
\fi
\bibliography{main}
\end{document}